\documentclass[11pt]{amsart}
\usepackage{graphicx}
\usepackage{amsmath}
\usepackage{amssymb} 
\usepackage{amsfonts}
\usepackage{amsthm}
\usepackage{mathtools}
\usepackage{biblatex}
\usepackage{tikz}

\newtheorem{thm}{Theorem}[section]
\newtheorem{lemma}{Lemma}[section]
\newtheorem{cor}{Corollary}[section]
\newtheorem{defn}{Definition}[section]
\newtheorem{prop}{Proposition}[section]

\newcommand{\C}{\mathbb{C}}
\newcommand{\Z}{\mathbb{Z}}

\newcommand{\Ha}{\mathcal{H}}

\newcommand{\g}{\mathfrak{g}}

\newcommand{\gl}{\mathfrak{gl}}
\newcommand{\p}{\mathfrak{p}}

\newcommand{\chr}{\text{char}}
\newcommand{\GL}{\textbf{GL}}

\newcommand{\Or}{\textbf{O}}

\newcommand{\Par}{\mathcal{P}}
\newcommand{\wt}{\text{wt}}

\newcommand{\CLR}{\text{CLR}}
\newcommand{\Hom}{\text{Hom}}

\title{Stable Graded Multiplicities for Harmonics on a Cyclic Quiver}

\author{Andrew Frohmader and Alexander Heaton}

\begin{document}

\maketitle

\begin{abstract}
    We consider Vinberg $\theta$-groups associated to a cyclic quiver on $k$ nodes. Let $K$ be the product of the general linear groups associated to each node. Then $K$ acts naturally on $\oplus \text{Hom}(V_i, V_{i+1})$ and by Vinberg's theory the polynomials are free over the invariants. We therefore consider the harmonics as a representation of $K$, and give a combinatorial formula for the stable graded multiplicity of each $K$-type. A key lemma provides a combinatorial separation of variables that allows us to cancel the invariants and obtain generalized exponents for the harmonics.
\end{abstract}

\section{Introduction}

Consider the representations of a cyclic quiver on $k$ nodes. Associate to each node a finite-dimensional vector space $V_j$, and to each arrow the space of linear transformations, $\mbox{Hom}(V_j, V_{j+1})$.  Set $V = V_1 \oplus \cdots \oplus V_k$ and let $K$ be the block diagonal subgroup of $G=\GL(V)$ isomorphic to $\GL(V_1) \times \cdots \times \GL(V_k)$ acting on
\begin{equation*}
    \mathfrak p = \mbox{Hom}(V_1,V_2) \oplus \mbox{Hom}(V_2,V_3) \oplus \cdots \oplus \mbox{Hom}(V_{k-1}, V_k) \oplus \mbox{Hom}(V_k, V_1).
\end{equation*}
Here we let $\GL(U) \times \GL(W)$ act on $\mbox{Hom}(U,W)$ by $(g_1,g_2) \cdot T  = g_2 \circ T \circ g_1^{-1}$, as usual. For $(T_1, \dots, T_k) \in \mathfrak p$, we have $K$-invariant functions defined by
\begin{equation*}
    \text{Trace}\left[ (T_1 \circ \cdots \circ T_k)^p \right]
\end{equation*}
for $1 \leq p \leq n = \min \{\dim V_j \}$. By a result of Le Bruyn and Procesi \cite{LeBryunProcesi1990}, these generate the $K$-invariant functions on $\mathfrak p$. The harmonic polynomials $\Ha$ are defined as the common kernel of all non-constant, $K$-invariant, constant-coefficient differential operators on $\p$. 

The harmonics are naturally graded by degree and we may encode the decomposition of $\Ha$ into $K$-irreducible representations by the $q$-graded character $\text{char}_q(\Ha)$, which places the character of the degree $d$ invariant subspace as the coefficient of $q^d$. If $s_K^\lambda$ is the irreducible character associated to the $K$-type $\lambda$, we may expand
\begin{equation*}
    \text{char}_q(\Ha) = \sum_\lambda m_\lambda^{(G,K)}(q) \, s_K^\lambda.
\end{equation*}
Fix the $K$-type $\nu$. Our main result is a combinatorial formula for $m_\nu^\infty(q,k)$, the stable multiplicity of $\nu$ in the harmonics on a cyclic quiver of length $k$. We will see that, for any cyclic quiver, $m_\nu^\infty(q,k)$ is equal to $m_\nu^{(G,K)}(q)$ up to degree $\leq n = \min \{ \dim V_j \}$ and our main Theorem \ref{thm:main} will prove
$$m_\nu^\infty(q,k) = \sum_{T \in D(\nu)} q^{\sum_{i=1}^k |\lambda_i(T)|}.$$
We sum over a certain set of \textit{distinguished tableau} $T \in \mathcal{D}(\nu)$, and the function $\lambda_i(T)$ is computed from simple combinatorial data associated to $T$. The proof consists of several steps. First, we realize the cyclic quiver above as a $\theta$-representation, or Vinberg pair $(G,K)$, with $K$ the fixed points of a finite order automorphism of $G$. A key lemma finds a combinatorial \textit{separation of variables}, mirroring Vinberg's theorem \cite{Vinberg1976} that
\begin{equation*}
    \C[\g_1] = \C[\g_1]^K \otimes \Ha.
\end{equation*}
Our combinatorial separation of variables allows us to cancel the invariants combinatorially. Other steps include constructing an action of a larger group and then restricting to $K$, applying a branching rule involving Littlewood-Richardson coefficients, and using the combinatorics of $\gl_\infty$ crystals to translate the $c^\lambda_{\mu,\nu}$ into tableau.

\section{Background}

\subsection{Vinberg Pairs}
Let $G$ be a connected reductive algebraic group over $\C$, and let $\theta:G\to G$ be an automorphism of $G$ with finite order $k$, so $\theta^k = \text{id}$. The group of fixed points $K = G^\theta$ acts on $\mathfrak{g}$ by restriction of the Adjoint representation. Each eigenspace of $d\theta$ is invariant. The Lie algebra splits into eigenspaces
\begin{equation*}
    \mathfrak{g} = \mathfrak{g}_0 \oplus \mathfrak{g}_1 \oplus \cdots \oplus \mathfrak{g}_{k-1}.
\end{equation*}
In \cite{Vinberg1976}, Vinberg studied the representation of $K$ on the polynomial functions on an eigenspace, and proved the following separation of variables:
\begin{equation*}
    \C[\mathfrak{g}_1] = \C[\mathfrak{g}_1]^K \otimes \mathcal{H},
\end{equation*}
where $\C[\mathfrak{g}_1]^K$ are the $K$-invariant functions and $\mathcal{H}$ are the harmonic polynomials. In general, for any representation of $K$ on $V$ the harmonics are defined as the common kernel for all invariant, non-constant, constant-coefficient differential operators $\mathcal{D}(V)^K$:
\begin{equation*}
    \mathcal{H} = \{ \, f \in \C[V] \, : \, \partial f = 0 \text{ for all non-constant } \partial \in \mathcal{D}(V)^K \}.
\end{equation*}

Note that with $k=1$, Vinberg's results recover those of Kostant's paper, \textit{Lie Group Representations on Polynomial Rings} \cite{Kostant1963}. There, Kostant proved the separation of variables
\begin{equation*}
    \C[\mathfrak{g}] = \C[\mathfrak{g}]^G \otimes \mathcal{H},
\end{equation*}
where $G$ acts on its Lie algebra under the Adjoint representation, $\C[\mathfrak{g}]^G$ are the invariants, and $\mathcal{H}$ are the harmonics. 

The harmonics are naturally graded by degree, and we may encode the decomposition of $\mathcal{H}$ into $G$-irreducible representations by the $q$-graded character $\text{char}_q(\mathcal{H})$, which places the character of the degree-$d$ invariant subspace as the coefficient of $q^d$. If $s^\lambda_G$ denotes the character of the $G$-irreducible representation parametrized by $\lambda$, then we have
\begin{equation*}
    \text{char}_q(\mathcal{H}) = \sum_\lambda \mathcal{K}^G_{\lambda,0}(q) \, s^\lambda_G.
\end{equation*}
In the Kostant setting, the polynomials $\mathcal{K}^G_{\lambda,0}(q)$ are called \textit{generalized exponents} of $G$ and coincide with the Lusztig $q$-analogues associated to the zero weight subspaces, by a theorem of Hesselink \cite{Hesselink1980}. Thus,
\begin{equation*}
    \mathcal{K}_{\lambda,0}^G(q) = \sum_{w \in W} (-1)^{l(w)}P_q(w(\lambda + \rho)-\rho),
\end{equation*}
where $\rho$ is the half sum of positive roots, $W$ is the Weyl group of $G$, $P_q$ is the $q$-Kostant partition function, and $l(w)$ is the length of $w \in W$.

Much work has been done in relation to these ideas, see \cite{KwonJang2021, LecouveyLenart2020, NelsenRam2003} and the references within.

The separation of variables above was generalized to the linear isotropy representation for a symmetric space by Kostant and Rallis \cite{KostantRallis1971}, and yet further to finite order automorphisms by Vinberg \cite{Vinberg1976}. Vinberg's work recovers the Kostant-Rallis results when $k=2$, which makes $\theta^2 = \text{id}$ an involution, and $(G,K)$ a symmetric pair. We may expand the $q$-graded character of the harmonics analogously in the Vinberg setting as
\begin{equation*}
    \text{char}_q(\mathcal{H}) = \sum_\lambda m^{(G,K)}_{\lambda}(q) \, s^\lambda_K.
\end{equation*}
The polynomials $m^{(G,K)}_{\lambda}(q)$ are much less understood. 

In the Kostant-Rallis setting, the graded multiplicities of an irreducible representation $\lambda$ in $\mathcal{H}$ may be described in terms of the eigenvalues of a certain element of $\mathfrak{k}$, see \cite[Theorem 21]{KostantRallis1971}. In \cite{WallachWillenbring2000}, Wallach and Willenbring obtain formulas similar to Hesselink for some examples including: $(GL_{2n}, Sp_{2n})$, $(SO_{2n+2},SO_{2n+1})$, and $(E_6, F_4)$. Wallach and Willenbring also worked out the example of $(SL_4, SO_4)$ explicitly and other results in special cases have appeared, \cite{WillenbringVanGroningen2014, JohnsonWallach1977}. There are also stable results stemming from the classical restriction rules of Littlewood \cite{HoweTanWillenbring2005, HoweTanWillenbring2008, Littlewood1944, Littlewood1945, Willenbring2002}. Recently, Frohmader developed a combinatorial formula for $(\GL_n, \Or_n)$ which is expected to generalize to the other classical symmetric pairs \cite{frohmader2023graded}. 

Moving outside of the Kostant-Rallis setting, even less is known. To our knowledge the only graded result is due to Heaton \cite{heaton2023graded}, in which he determines the graded multiplicity for $(\GL_{2r}, \GL_2 \times \dots \times \GL_2)$ by counting integral points on the intersection of polyhedra. Wallach has developed ungraded multiplicity formulas, see \cite{Wallach2017, Wallach2017-2}. Our contribution is a stable formula for $m_\lambda^{(G, K)}(q)$ for $(G,K) = (\GL_N, \GL_{n_1} \times \dots \times \GL_{n_k})$, where $N = \sum_{i=1}^k n_i$.

\subsection{Partitions, Tableaux, and $\GL_n$ representations }

For a partition $\lambda$, let $l(\lambda)$ denote length($\lambda)$ and $|\lambda|$ the size (number of boxes) of $\lambda$. Let $\Par_n$ denote the set of partitions with length $\leq n$ (including the empty partition $\varnothing$) and $\Par$ the set of all partitions. Two bases are useful in discussing irreducible polynomial representations of $\GL_n$: $\epsilon_1, \dots , \epsilon_n$  and $\omega_1, \dots , \omega_n$, where $\omega_i = \epsilon_1 + \epsilon_2 + \dots + \epsilon_i$. The polynomial representations of $\GL_n$ are in one to one correspondence with highest weights $\lambda = a_1\epsilon_1 + \dots + a_n \epsilon_n$, where $a_1 \geq a_2 \geq \dots \geq a_n \geq 0$ are non-negative integers. This gives a bijection between partitions and irreducible polynomial $\GL_n$ representations. In terms of the $\omega_i$ basis, the highest weights are given by $\lambda = b_1 \omega_1 + \dots + b_n \omega_n$ where all $b_i \in \Z_{\geq 0}$. There are no order conditions. So the $\omega_i$ basis allows us to identify irreducible polynomial $\GL_n$ representations with n-tuples of non-negative integers. Computing the change of basis matrices, we see
$$\lambda = (a_1 - a_2)\omega_1 + \dots (a_{n-1} - a_n)\omega_{n-1} + a_n \omega_n$$
$$\lambda = (b_1 + \dots + b_n) \epsilon_1 + (b_2 + \dots + b_n) \epsilon_2 + \dots + (b_{n-1} + b_n)\epsilon_{n-1} + b_n \epsilon_n$$
In terms of partitions, $\epsilon_i$ corresponds to a box in row $i$ and $\omega_i$ corresponds to a column of length $i$.

Define a partial order on $\Par$ by $\mu \leq \lambda$ if $\lambda - \mu \in \Par$. In what follows, it will be helpful to view the product order on $\Z^\infty = \{(a_1, a_2, \dots) \ : \ a_i \in \Z \text{ and } a_i = 0 \text{ for all but finitely many } i \}$, as extending $\leq$. Recall this is the order $(b_1, b_2, \dots ) \leq (a_1, a_1, \dots )$ if and only if $a_i - b_i \in \Z_{\geq 0}$ for all $i$. To accomplish this, write $\lambda = a_1 \omega_1 + \dots + a_n \omega_n$ and $\mu = b_1 \omega_1 + \dots + b_n \omega_n$ in terms of the $\omega_i$ basis. Notice $\lambda - \mu \in \Par$ if and only if $(a_1, \dots , a_n, 0, 0, \dots) - (b_1, \dots , b_n, 0, 0, \dots) \in \Z^\infty_{\geq 0}$ if and only if $a_i - b_i \in \Z_{\geq 0}$. 

Let $SST_n(\lambda)$ be the set of semistandard tableaux on $\lambda$ with entries in $\{1, \dots , n\}$ and $SST(\lambda)$ the set of semistandard tableaux on $\lambda$ with entries in $\Z_{>0}$. We view $SST_n(\lambda)$ and $SST(\lambda)$ as $\gl_n$ and $\gl_\infty$ crystals, see \cite{BumpSchilling2017, HongKang2002}. Define the weight of a tableau $T \in SST(\lambda)$ by $\wt(T) = k_1 \epsilon_1 + \dots + k_n \epsilon_n$ where $k_i$ denotes the number of $i$'s appearing in $T$. Writing $\wt(T)$ in terms of the $\omega_i$, we see the reason for extending $\leq$ to $Z^\infty$ is to enable comparison with non-dominant weights. For example, given $T$ a tableau on a one-box shape with content 2, $\wt(T) = \epsilon_2 = - \omega_1 + \omega_2$ which we identify with $(-1, 1, 0,0, \dots)$.   

\section{The action of $K^2$}
We have an action of $K = \GL_{n_1} \times \GL_{n_2} \times  \dots \times \GL_{n_k}$ on $\p = M_{n_2,n_1} \oplus M_{n_3,n_2} \oplus \dots \oplus M_{n_k,n_{k-1}} \oplus M_{n_1,n_k}$ by 
$$(g_1,g_2, \dots ,g_k) \cdot (X_1,X_2, \dots , X_k) = (g_2 X_1 g_1^{-1}, g_3 X_2 g_2^{-1} \dots g_1 X_k g_k^{-1}).$$

\begin{figure}
    \centering
    \begin{tikzpicture}
        \draw[thick,dashed] (-2.4,-1.4) arc (180:150:1cm);
        \draw[thick,->] (-1.7,-0.3) node[anchor=north east] {$\C^{n_k}$} arc (135:90:1cm) node[anchor=south east] {$M_{n_1,n_k}$};
        \draw[thick,->] (0,0) node[anchor= east] {$\C^{n_1}$} node[anchor=south west] {$M_{n_2,n_1}$} arc (90:45:1cm) node[anchor=north west] {$\C^{n_2}$};
        \draw[thick,->] (1.2,-1) arc (45:0:1cm) node[anchor=south west] {$M_{n_3,n_2}$} node[anchor=north] {$\C^{n_3}$};
        \draw[thick,dashed] (1.5,-2.3) arc (0:-30:1cm);
    \end{tikzpicture}
    \caption{Cyclic quiver on $k$ nodes}
    \label{fig:cyclic-quiver}
\end{figure}
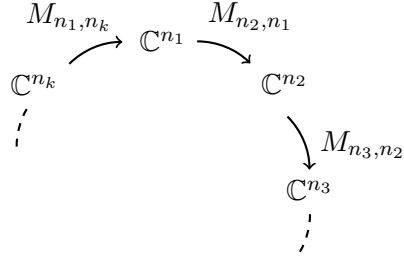

\noindent This yields an action of $K$ on $\C[\p]$, $k \cdot f(X) = f(k^{-1}\cdot X)$ for $k \in K$ and $X \in \p$. We would like to understand the graded multiplicities of this action. (Notice the indices are cyclically permuted, as in Figure \ref{fig:cyclic-quiver}).

We can approach the problem through branching starting from the action of $K^2 = \GL_{n_1}^2 \times \GL_{n_2}^2 \times  \dots \times \GL_{n_k}^2$ on $\C[\p]$  by 
$$(g_1,h_1, \dots ,g_k, h_k) \cdot f(X_1,X_2, \dots , X_k) = f(g_2^{-1} X_1 h_1, \dots g_1^{-1} X_k h_k).$$

\noindent Here $\GL_{n_i}^2$ denotes $\GL_{n_i} \times \GL_{n_i}$. Of course, we want to restrict this action to the diagonal subgroup $\Delta = \{(g_1,g_1,g_2,g_2, \dots , g_k,g_k) \} \cong K$. So we have two tasks: first understand the representation of the big group $K^2$, second understand how this representation restricts to $\Delta$. 

We begin by determining the of $K^2$ irreducible representations in $\C[\p]$. First, recall,
    
    \begin{prop}\cite[Proposition 4.2.5]{GoodmanWallach2009} The irreducible representations of $\GL_{n_1} \times \GL_{n_2} \times \dots \times \GL_{n_l}$ are the representations $V_1 \otimes V_2 \otimes \dots \otimes V_l$ where $V_i$ is an irreducible representation of $\GL_{n_i}$.
    
    \end{prop}
        
Next, notice $\C[\p] = \C[M_{n_2,n_1} \oplus \dots \oplus M_{n_1,n_k}] \cong \C[M_{n_2,n_1} ]\otimes \dots \otimes \C[M_{n_1,n_k}]$, see \cite[Lemma A.1.9]{GoodmanWallach2009} and we have $k$ commuting actions. For example, $\GL_{n_2} \times \GL_{n_1}$ acts by 
    
    $$f_1(X_1) \otimes \dots \otimes f_k(X_k)\to f_1(g_2^{-1}X_1 h_1) \otimes \dots \otimes f_k(X_k).$$

\noindent In fact, we can recognize this representation as the tensor product of $k$ distinct actions, so we can decompose the actions separately. 

Now recall,

    \begin{thm}\label{thm-multfreedecomp}\cite[Theorem 5.6.7]{GoodmanWallach2009} The degree $d$ component of $\C[M_{n_i,n_j}]$ under the action of $\GL_{n_i} \times \GL_{n_j}$ decomposes as follows
    
    $$\C^d[M_{n_i,n_j}] \cong \bigoplus_\lambda (F_{n_i}^\lambda)^*\otimes (F_{n_j}^\lambda)$$

    \noindent with the sum over all nonnegative dominant weights $\lambda$ of size $d$ and length $\leq \min\{n_i,n_j\}$.
    \end{thm}

Hence we have the following graded decomposition of the $K^2$ representation (Note: In all that follows we consider our indexing with respect to the cyclic quiver, i.e. mod k with representatives 1, 2, ..., k):

    \begin{thm} \label{thm-biggroupdecomp}
        The degree $d$ component of $\C[\p]$ under the action of $K^2$ decomposes as follows 

        $$\C[\p] \cong \bigoplus_{\lambda_1, \lambda_2, \dots , \lambda_k} [(F_{n_2}^{\lambda_1})^*\otimes F_{n_1}^{\lambda_1}] \otimes [(F_{n_3}^{\lambda_2})^*\otimes F_{n_2}^{\lambda_2}] \otimes \dots \otimes [(F_{n_1}^{\lambda_k})^*\otimes F_{n_k}^{\lambda_k}]$$

        $$\cong \bigoplus_{\lambda_1, \lambda_2, \dots , \lambda_k} \bigotimes_{i=1}^k [(F_{n_{i+1}}^{\lambda_i})^* \otimes F_{n_i}^{\lambda_i}]$$

        \noindent with the sum over all nonnegative dominant weights $\lambda_1, \lambda_2, \dots ,\lambda_k$ such that $|\lambda_1| + |\lambda_2| + \dots + |\lambda_k| = d$ and length($\lambda_i) \leq \min\{n_i,n_{i+1}\}$.
    \end{thm}

    \begin{proof}
        As discussed above, we can decompose each $\C^d[M_{n_i,n_j}]$ factor separately. Apply Theorem \ref{thm-multfreedecomp}.
    \end{proof}

With the action of $K^2$ understood, we turn to the problem of branching to the diagonal subgroup $\Delta$.

\section{Stable Multiplicities via Branching}

Let $n = \min \{n_1, \dots , n_k\}$. We work with the pairs $\GL_{n_1}^2$, $\GL_{n_2}^2$, ... , $\GL_{n_k}^2$ separately. Essentially, we choose to group the decomposition from Theorem \ref{thm-biggroupdecomp} as

$$\bigoplus_{\lambda_1, \lambda_2, \dots , \lambda_k} [F_{n_1}^{\lambda_1}\otimes (F_{n_1}^{\lambda_k})^*] \otimes [F_{n_2}^{\lambda_2}\otimes (F_{n_2}^{\lambda_1})^*] \otimes \dots \otimes [F_{n_k}^{\lambda_k}\otimes (F_{n_k}^{\lambda_{k-1}})^*]$$

$$\cong \bigoplus_{\lambda_1, \lambda_2, \dots , \lambda_k} \bigotimes_{i=1}^k [F_{n_i}^{\lambda_i}\otimes (F_{n_i}^{\lambda_{i-1}})^*]$$

Recall,

\begin{thm} [Stable Branching Rule] \cite[Theorem 2.1.4.1]{HoweTanWillenbring2008} \label{thm-stablebranching}

For $l(\lambda_i) + l(\lambda_{i-1}) \leq n_i$,

$$\dim \Hom_{\GL_{n_i}}(F^{\nu^+, \nu^-}_{n_i}, F^{\lambda_i}_{n_i} \otimes (F^{\lambda_{i-1}}_{n_i})^*) = \sum_\alpha c^{\lambda_i}_{\alpha, \nu^+} c^{\lambda_{i-1}}_{\alpha, \nu^-}.$$

\end{thm}

$F_{n_i}^{\nu^+, \nu^-}$ is our notation for the rational representation of $\GL_{n_i}$ corresponding to the tuple of partitions $(\nu^+, \nu^-)$. Both $\nu^+$ and $\nu^-$ are partitions, and if $\nu^+ = (a_1,a_2,\dots,a_\ell)$ and $\nu^- = (b_1, b_2, \dots, b_m)$ then $F_{n_i}^{\nu^+, \nu^-}$ is the rational representation of $\GL_{n_i}$ with highest weight $$(a_1,a_2,\dots,a_\ell,0,\dots,0,-b_m,-b_{m-1},\dots,-b_2,-b_1),$$ with the number of interior zeros arranged appropriately, see \cite{STEMBRIDGE198779}. Hence we have,

\begin{thm}
    For degree $d \leq n$, the degree $d$ component of $\C[\p]$ under the action of $K$ decomposes as follows,

    $$\bigoplus_{\alpha_i, \lambda_i, \nu_i^\pm}\bigotimes_{i=1}^k
    c^{\lambda_i}_{\alpha_i, \nu_i^+} c^{\lambda_{i-1}}_{\alpha_i, \nu_i^-} 
    F^{\nu_i^+, \nu_i^-}_{n_i}$$

    \noindent with the sum over all $\{\alpha_i, \lambda_i, \nu_i^\pm\}_{i=1}^k$ in $\Par_n$ such that $|\lambda_1| + \dots + |\lambda_k| = d$. In particular, the multiplicity of the $K$ irrep $\nu = (\nu_1^\pm, \dots , \nu_k^\pm)$ appearing in degree $d$ is given by

    $$\sum_{\alpha_i, \lambda_i} (\prod_{i=1}^k
    c^{\lambda_i}_{\alpha_i, \nu_i^+} c^{\lambda_{i-1}}_{\alpha_i, \nu_i^-}).$$
    
\end{thm}

\begin{proof}
    Say $d \leq n$. Then for any $\lambda_i, \lambda_{i-1}$, $l(\lambda_i) + l(\lambda_{i-1}) \leq |\lambda_i| + |\lambda_{i-1}| \leq d \leq n \leq n_i$ so Theorem \ref{thm-stablebranching} applies and we understand the branching down to $K$. We also note it suffices to consider partitions in $\Par_n$ since if a partition $\alpha_i, \lambda_i$ or $\nu_i^\pm$ has length greater than $n$, it contributes to a degree greater than $n$ and so only impacts multiplicities outside the stable range.
\end{proof}

\begin{cor}
    The following gives the graded character $\chr_q(\C[\p])$ up to degree $n$,

    $$\sum_{\alpha_i, \lambda_i, \nu_i^\pm} q^{\sum |\lambda_i|}\prod_{i=1}^k
    c^{\lambda_i}_{\alpha_i, \nu_i^+} c^{\lambda_{i-1}}_{\alpha_i, \nu_i^-} 
    s^{\nu_i^+, \nu_i^-}_{n_i}$$

    \noindent where $s^{\nu_i^+, \nu_i^-}_{n_i}$ is the $\GL_{n_i}$ character of $F^{\nu_i^+, \nu_i^-}_{n_i}$ and the sum is taken over all $\{\alpha_i, \lambda_i, \nu_i^\pm\}_{i=1}^k$ in $\Par_n$.
    
\end{cor}

Next, we handle the invariants, which are generated by $Tr([X_1 X_2 \dots X_k]^i)$ for $1 \leq i \leq n$ by a result in \cite{LeBryunProcesi1990}.
\begin{prop}
    We have the separation of variables
    \begin{equation*}
        \C[\p] = \C[\p]^K \otimes \Ha.
    \end{equation*}
\end{prop}
\begin{proof}
    Notice that $K = G^\theta$ where $\theta:G \to G$ is given by conjugation by the diagonal matrix with entries equal to $k$th roots of unity $1,\zeta,\zeta^2,\dots,\zeta^{k-1}$, each appearing with multiplicities $n_1,\dots,n_k$. The conjugation action of $K$ on the $\zeta$-eigenspace is isomorphic to the action of $K$ on $\p$. The result now follows from Vinberg's theory \cite{Vinberg1976}.
\end{proof}
 
Hence, the graded character of $\Ha$ is given by

$$\chr_q(\Ha) = [\prod_{i=1}^n(1-q^{ki})]\chr_q(\C[\p])$$.

\begin{cor} \label{cor-gradedHchar}
    The following gives the graded character $\chr_q(\Ha)$ up to degree $n$,

    $$[\prod_{i=1}^n(1-q^{ki})]\sum_{\alpha_i, \lambda_i, \nu_i^\pm} q^{\sum |\lambda_i|}\prod_{i=1}^k
    c^{\lambda_i}_{\alpha_i, \nu_i^+} c^{\lambda_{i-1}}_{\alpha_i, \nu_i^-} 
    s^{\nu_i^+, \nu_i^-}_{n_i}$$

    \noindent where $s^{\nu_i^+, \nu_i^-}_{n_i}$ is the $\GL_{n_i}$ character of $F^{\nu_i^+, \nu_i^-}_{n_i}$ and the sum is taken over all partitions $\{\alpha_i, \lambda_i, \nu_i^\pm\}_{i=1}^k$ in $\Par_n$.

    In particular, the following formula provides the graded multiplicity of the $K$ irrep $\nu = (\nu_1^\pm, \dots , \nu_k^\pm)$ in $\Ha$, denoted $m_{\nu}(q)$, up to degree $n$,

    $$[\prod_{i=1}^n(1-q^{ki})]\sum_{\alpha_i, \lambda_i} q^{\sum |\lambda_i|}\prod_{i=1}^k
    c^{\lambda_i}_{\alpha_i, \nu_i^+} c^{\lambda_{i-1}}_{\alpha_i, \nu_i^-}.$$
\end{cor}

\begin{proof}
    Immediate from above discussion.
\end{proof}

\begin{cor}
    For $\nu$ a $K$ irrep, if $\sum_{i=1}^k |\nu_i^+| > n$ or if $\sum_{i=1}^k |\nu_i^-| > n$, then $m_\nu(q) = 0$ in the stable range.
\end{cor}

\begin{proof}
    Notice in the formula of Corollary \ref{cor-gradedHchar}, the smallest degrees come from the $q^{\sum |\lambda_i|}$ terms. Now, by basic properties of Littlewood-Richardson coefficients, if the term $q^{\sum |\lambda_i|}\prod_{i=1}^k
    c^{\lambda_i}_{\alpha_i, \nu_i^+} c^{\lambda_{i-1}}_{\alpha_i, \nu_i^-}$ is not zero, $|\lambda_i| \geq |\nu_i^+|$ for all $i$, but then $q^{\sum |\lambda_i|} \geq q^{\sum |\nu_i^+|} > q^{n}$. So $m_\nu(q)$ is 0 in degree less than or equal to $n$.
\end{proof}

We now turn our attention to stable multiplicities and make the following key definition.
\begin{defn}
    $$m_\nu^\infty(q,k) = [\prod_{i=1}^\infty(1-q^{ki})]\sum_{\alpha_i, \lambda_i} q^{\sum |\lambda_i|}\prod_{i=1}^k
    c^{\lambda_i}_{\alpha_i, \nu_i^+} c^{\lambda_{i-1}}_{\alpha_i, \nu_i^-},$$
    
\noindent where the sum is taken over all partitions $\{\alpha_i, \lambda_i\}_{i=1}^k$ in $\Par$. This is the stable $q$-multiplicity for $\nu$ on a quiver of length $k$.
\end{defn}

It is easy to see that $m_\nu^\infty(q,k) = m_\nu(q)$ up to degree $n$. These stable $q$-multiplicities will be our focus for the remainder of the paper.

We would like to cancel the $[\prod_{i=1}^\infty(1-q^{ki})]$ factor from the formula for $m_\nu^\infty(q,k)$. We recall, see \cite{frohmader2023graded} for example, that $c^{\lambda}_{\alpha, \nu} = |\CLR^\lambda_{\alpha, \nu}| := |\{T \in SST(\nu) \ | \ \alpha \geq \varepsilon(T) \text{ and } \alpha + \wt(T) = \lambda \}|$. Here we are viewing $SST(\nu)$ as a $\gl_\infty$ crystal with Kashiwara operators $\tilde{e}_i$ and $\tilde{f}_i$ for $i = 1, 2, \dots $ and we define $\varepsilon_i(T) = \max \{ k \geq 0 \ | \ \tilde{e}_i^k T \in SST(\lambda) \}$, $\phi_i(T) = \max \{ k \geq 0 \ | \ \tilde{f}_i^k T \in SST(\lambda) \}$, and 

\begin{align*}
\phi(T) &= \sum_{i=1}^{n-1}\phi_i(T)\omega_i,   &   \varepsilon(T) &= \sum_{i=1}^{n-1}\varepsilon_i(T)\omega_i.    
\end{align*}

In this notation, we have,

$$m_\nu^\infty(q,k) = [\prod_{i=1}^\infty(1-q^{ki})]\sum_{\alpha_i, \lambda_i} q^{\sum |\lambda_i|}\prod_{i=1}^k
    |\CLR^{\lambda_i}_{\alpha_i, \nu_i^+}| |\CLR^{\lambda_{i-1}}_{\alpha_i, \nu_i^-}|.$$

Notice the formula for $m_\nu^\infty(q,k)$ has $\nu_i^\pm$ fixed for all $i$, so we are just computing various subsets of $\underbar{SST}(\nu) := \prod_{i=1}^k [SST(\nu_i^+) \times SST(\nu_i^-)]$. The key is to understand which $T = (T_1^+,T_1^-, \dots , T_k^+, T_k^-)\in \underbar{SST}(\nu)$ appear in some $\underbar{\CLR}^\lambda_{\alpha, \nu} := \prod_{i=1}^k \CLR^{\lambda_i}_{\alpha_i, \nu_i^+} \times \CLR^{\lambda_{i-1}}_{\alpha_i, \nu_i^-}$ and with what multiplicity. In this context, $\lambda = (\lambda_1, \dots , \lambda_k), \alpha = (\alpha_1, \dots \alpha_k),$ and $\nu = (\nu_1^\pm, \dots, \nu_k^\pm)$ are tuples of partitions.

As $T_i = (T_i^+, T_i^-)$ is associated with the rational $\GL_{n_i}$ representation $F^{\nu_i^+, \nu_i^-}_{n_i}$, let $\wt(T_i) := \wt(T_i^+) - \wt(T_i^-)$. Also denote the set of all $k$-tuples of tableaux $\Par^k$. We first isolate those $T$ contributing with 

\begin{defn}
    A tuple of tableaux $T\in \underbar{SST}(\nu)$ is called distinguished if $T \in \underbar{\CLR}^\lambda_{\alpha, \nu}$ for some $\lambda, \alpha \in \Par^k$.
\end{defn}

\begin{defn}
    Let $D(\nu)$ be the set of all distinguished tableaux in $\underbar{SST}(\nu)$.
\end{defn}

\begin{lemma} \label{lemma-alphalambda}
    Suppose $T\in \underbar{\CLR}^\lambda_{\alpha, \nu}$. Then,

    $$\alpha_i = \lambda_1 - \wt(T_i^+) + \sum_{j=2}^i \wt(T_j),$$
    $$\lambda_i = \lambda_1 + \sum_{j=2}^{i} \wt(T_j),$$
    $$\sum_{j=1}^k \wt(T_j) = 0,$$

    for all $\alpha_i$ and $\lambda_i$. In particular, $\alpha_i$ and $\lambda_i$ are uniquely determined by $\lambda_1$ and $T$. With $T$ fixed, let $\lambda(\lambda_1)$ and $\alpha(\lambda_1)$ be those elements of $\Par^k$ determined by $\lambda_1$.
\end{lemma}

\begin{proof}
    We begin by establishing the formula for $\lambda_i$. Proceed by induction. The base case is clear. Now assume the formula holds for $\lambda_{i-1}$ with $1 < i \leq k$. From the term $\CLR^{\lambda_i}_{\alpha_i, \nu_i^+} \times \CLR^{\lambda_{i-1}}_{\alpha_i, \nu_i^-}$ we see $\alpha_i = \lambda_{i-1} - \wt(T_i^-)$ so by induction, $\alpha_i = \lambda_1 + \sum_{j=2}^{i-1}\wt(T_j) - \wt(T_i^-)$ and $\lambda_i = \alpha_i + \wt(T_i^+) = \lambda_1 + \sum_{j=2}^i\wt(T_j)$.

    Next, we establish the third equality. We have, 
    
    $$\lambda_k = \lambda_1 + \sum_{j=2}^{k} \wt(T_j)$$
    
    Notice from the $ \CLR^{\lambda_1}_{\alpha_1, \nu_1^+} \times \CLR^{\lambda_{k}}_{\alpha_1, \nu_1^-}$ factor, we also have,

    $$\lambda_k = \wt(T_1^-) + \alpha_1.$$
    $$\lambda_1 = \wt(T_1^+) + \alpha_1.$$

    Subtracting the two expressions for $\lambda_k$ yields

    $$\sum_{j=1}^k \wt(T_j) = 0.$$

    Finally, from the term $\CLR^{\lambda_i}_{\alpha_i, \nu_i^+}$ we see $\alpha_i = \lambda_i - \wt(T_i^+) = \lambda_1 + \sum_{j=2}^{i} \wt(T_j) - \wt(T_i^+)$. 
    
\end{proof}

Lemma \ref{lemma-alphalambda} shows that with $T$ fixed, there is at most a 1-parameter family of $\underbar{\CLR}^\lambda_{\alpha, \nu}$ containing $T$. We choose to parameterize this family by $\lambda_1$, but note that any choice of a fixed $\lambda_i$ or $\alpha_j$ uniquely constrains  $\underbar{\CLR}^\lambda_{\alpha, \nu}$ and could be used as parameter. The lemma below shows that the cyclic nature of the representation constrains the set of distinguished tableaux.

\begin{lemma} \label{lemma-distingchar}
    $T \in \underbar{SST}(\nu)$ is distinguished if and only if $\sum_{j=1}^k \wt(T_j) = 0$.
\end{lemma}

\begin{proof}
    Say $T$ is distinguished. Then $T \in \underbar{\CLR}^\lambda_{\alpha, \nu}$ for some $\lambda$ and $\alpha$ so by Lemma \ref{lemma-alphalambda}, $0 = \sum_{j=1}^{k} \wt(T_j)$.

    Now suppose $\sum_{i=1}^k \wt(T_i) = 0$. We must show $T \in \underbar{\CLR}^\lambda_{\alpha, \nu}$ for some $\lambda, \alpha \in \Par^k$. To do this, we require two things. First, $\alpha_i \geq \varepsilon(T_i^+)$ and $\alpha_i \geq \varepsilon(T_i^-)$ for all $i$. This ensures $T_i^+ \in \CLR_{\alpha_i, \nu_i^+}^{\alpha_i+\wt(T_i^+)}$ and similarly for $T_i^-$. Second, we have to make sure the $\lambda_i$ are compatible, that is the two formulas for $\lambda_i$, $\lambda_i = \alpha_i + \wt(T_i^+)$ and $\lambda_i = \alpha_{i+1}+ \wt(T_{i+1}^-)$ are equal.

    By Lemma \ref{lemma-alphalambda}, to achieve $\lambda_i$ compatibility, we must have $\lambda_i = \lambda_1 + \sum_{j=2}^{i} \wt(T_j)$ for $i > 1$ and $\sum_{j=1}^k \wt(T_j) = 0$, i.e. we are constrained to work within the family parameterized by $\lambda_1$. The proof comes down to showing this family is not empty by selecting a $\lambda_1$ large enough that $\alpha_i \geq \varepsilon(T_i^+)$ and $\alpha_i \geq \varepsilon(T_i^-)$ for all $i$. As $\alpha_i = \lambda_1 - \wt(T_i^+) + \sum_{j=2}^{i} \wt(T_j)$, this can be achieved by selecting $\lambda_1 \geq \sup \{\varepsilon(T_i^\pm) + \wt(T_i^+) - \sum_{j=2}^{i} \wt(T_j)\}_{i=1}^k$. Indeed, then 
    
    $$\alpha_i = \lambda_1 - \wt(T_i^+) + \sum_{j=2}^{i} \wt(T_j)$$

    $$\geq [\varepsilon(T_i^\pm) + \wt(T_i^+) - \sum_{j=2}^{i} \wt(T_j)] - \wt(T_i^+) + \sum_{j=2}^{i} \wt(T_j)$$
    $$= \varepsilon(T_i^\pm).$$
    
    So we have $\alpha_i \geq \varepsilon(T_i^\pm)$, which shows $\lambda_i$ and $\alpha_i$ are partitions and hence $T$ is contained in $\underbar{\CLR}^\lambda_{\alpha, \nu}$.

    \end{proof}

\begin{defn}
    We isolate a least upper bound from the proof of Lemma \ref{lemma-distingchar} in this definition. For $T \in D(\nu)$ define $\lambda_{\min}(T) = \sup \{\varepsilon(T_i^\pm) + \wt(T_i^+) - \sum_{j=2}^{i} \wt(T_j)\}_{i=1}^k$.
\end{defn}

$\lambda_{\min}$ exists. It can be explicitly constructed as follows. Notice we can work in $\Par_N$ if we choose $N$ large enough. Writing each $S^\pm_i = \varepsilon(T_i^\pm) + \wt(T_i^+) - \sum_{j=2}^{i} \wt(T_j)$ in terms of the $\omega_i$ basis as $S^\pm_j = a^\pm_{1 j}\omega_1 = \dots + a^\pm_{N j} \omega_{N}$. Set $a_i = \max \{a^\pm_{i 1}, \dots , a^\pm_{i k} \}$, that is $a_i$ is the maximum coefficient of $\omega_i$ across the $S^\pm_i$. Then $\lambda_{\min}(T) = a_1 \omega_1 + \dots + a_{N} \omega_{N}$. Notice also that $S_1^+ = \varepsilon(T_1^+) + \wt(T_1^+)$ is a partition by the tensor product rule for crystals, that is $S_1^+ = a^+_{11} \omega_1 + \dots + a^+_{N 1} \omega_N$ with $a_{1i} \in \Z_{\geq 0}$ for all $i$. Hence, $a_i \geq 0$ for all $i$.

Next, we give a name to the set of partitions parameterizing the $\underbar{\CLR}^\lambda_{\alpha, \nu}$ containing T. 

\begin{defn}
    For $T \in D(\nu)$ let $S_T$ be the set of all $\lambda_1 \in \Par$ such that $T \in \underbar{\CLR}^{\lambda(\lambda_1)}_{\alpha(\lambda_1), \nu}$.

\end{defn}

\begin{lemma}\label{lemma-TinCLR}
    For $T \in D(\nu)$, $T \in \underbar{\CLR}^{\lambda(\lambda_1)}_{\alpha(\lambda_1), \nu}$ if and only if $\lambda_1 \geq \lambda_{\min}(T)$.
\end{lemma}

\begin{proof}
    This follows from the proof of Lemma \ref{lemma-distingchar}.
\end{proof}

\begin{lemma}
    For $T \in D(\nu)$, $S_T = \lambda_{\min}(T) + \Par.$
\end{lemma}

\begin{proof}
    This follows from Lemma \ref{lemma-TinCLR} by observing that $\lambda_{\min}(T)$ is the unique minimal element in $S_T$ so for any $\delta \in S_T$ we can write $\delta = \lambda_{\min}(T) + (\delta - \lambda_{\min}(T))$.
\end{proof}

Hence, for $T \in D(\nu)$, we have a 1-parameter family of $\underbar{\CLR}^\lambda_{\alpha, \nu}$ containing $T$, now parameterized by $\delta \in \Par$. We define the following functions

$$\lambda_i(T, \delta) = \lambda_{\min}(T) + \delta + \sum_{j=2}^{i} \wt(T_j),$$
$$\alpha_i(T, \delta) = \lambda_{\min}(T) + \delta - \wt(T_i^+) + \sum_{j=2}^{i} \wt(T_j).$$

Then this family can be written explicitly as 

$$ \{\prod_{i=1}^k \CLR^{\lambda_i(T, \delta)}_{\alpha_i(T, \delta), \nu_i^+} \times \CLR^{\lambda_{i-1}(T,\delta)}_{\alpha_i(T,\delta), \nu_i^-} \ : \  \delta \in \Par \} $$

\begin{lemma} \label{lemma-lambdaiadditive}
    $$\lambda_i(T, \delta) = \lambda_i(T, \varnothing) + \delta$$
\end{lemma}
\begin{proof}
    $$\lambda_i(T, \delta) = \lambda_{\min}(T) + \delta + \sum_{j=2}^{i} \wt(T_j),$$
    $$= \lambda_{\min}(T) + \varnothing + \sum_{j=2}^{i} \wt(T_j) + \delta,$$
    $$ = \lambda_i(T, \varnothing) + \delta.$$
\end{proof}

Denote $\lambda_i(T, \varnothing)$ by $\lambda_i(T)$ for simplicity. We isolate the following key lemma which should be viewed as a combinatorial separation of variables. 

\begin{lemma}
    $$\frac{1}{\prod_{i=1}^\infty(1-q^{ki})} m_\nu^\infty(q,k) = \sum_{\delta \in \Par} q^{k|\delta|} \sum_{T \in D(\nu)} q^{\sum_{i=1}^k |\lambda_i(T)|}$$
\end{lemma}

\begin{proof}

    $$\frac{1}{\prod_{i=1}^\infty(1-q^{ki})}m_\nu^\infty(q,k) = \sum_{T \in D(\nu)} \sum_{\delta \in \Par} q^{\sum_{i=1}^k |\lambda_i(T, \delta)|}$$
    
Now by Lemma \ref{lemma-lambdaiadditive}, 

    $$ = \sum_{T \in D(\nu)} \sum_{\delta \in \Par} q^{\sum_{i=1}^k |\lambda_i(T, \varnothing) + \delta|}$$

    $$= \sum_{T \in D(\nu)} \sum_{\delta \in \Par} q^{k|\delta|\sum_{i=1}^k |\lambda_i(T, \varnothing)|}$$

    $$= \sum_{\delta \in \Par} q^{k|\delta|} \sum_{T \in D(\nu)} q^{\sum_{i=1}^k |\lambda_i(T, \varnothing)|}$$
\end{proof}

From this, the main theorem is immediate. Cancel $\sum_{\delta \in \Par} q^{k|\delta|}$ with the invariants $1/\prod_{i=1}^\infty(1-q^{ki})$.

\begin{thm}\label{thm:main}
     $$m_\nu^\infty(q,k) = \sum_{T \in D(\nu)} q^{\sum_{i=1}^k |\lambda_i(T)|}$$
\end{thm}

\printbibliography

@book {BumpSchilling2017,
    AUTHOR = {Bump, Daniel and Schilling, Anne},
     TITLE = {Crystal bases},
      NOTE = {Representations and combinatorics},
 PUBLISHER = {World Scientific Publishing Co. Pte. Ltd., Hackensack, NJ},
      YEAR = {2017},
     PAGES = {xii+279},
}

@article{Vinberg1976,
    author = "Vinberg, Ernest B.",
    title = "The Weyl group of a graded Lie algebra",
    journal = "Mathematics of the USSR-Izvestiya",
    volume = "10",
    number = "3",
    pages = "463",
    year = "1976"
}

@article{LeBryunProcesi1990,
    author = "Le Bryun, Lieven and Procesi, Claudio",
    title = "Semisimple representations of quivers",
    journal = "Transactions of the American Math Society",
    volume = "317",
    number = "2",
    pages = "585-598",
    year = "1990"
}

@misc{frohmader2023graded,
      title={Graded Multiplicities in the Kostant-Rallis Setting}, 
      author={Andrew Frohmader},
      year={2023},
      eprint={2312.11295},
      archivePrefix={arXiv},
      primaryClass={math.RT}
}

@book{GoodmanWallach2009, place={New York}, title={Symmetry, Representations, and Invariants}, publisher={Springer}, author={Goodman, Roe and Wallach, Nolan}, year={2009}, series={Graduate Texts in Mathematics}}

@misc{heaton2023graded,
      title={Graded multiplicity in harmonic polynomials from the Vinberg setting}, 
      author={Alexander Heaton},
      year={2023},
      eprint={1805.03178},
      archivePrefix={arXiv},
      primaryClass={math.RT}
}

@article{Hesselink1980,
  title = "Characters of the nullcone",
  journal = "Mathematische Annalen",
  volume = "252",
  pages = "179-182",
  year = "1980",
  author = "Hesselink, Wim H."
}

@book{HongKang2002,
    AUTHOR = {Hong, Jin and Kang, Seok-Jin},
     TITLE = {Introduction to quantum groups and crystal bases},
    SERIES = {Graduate Studies in Mathematics},
    VOLUME = {42},
 PUBLISHER = {American Mathematical Society, Providence, RI},
      YEAR = {2002},
     PAGES = {xviii+307},
  MRNUMBER = {1881971}
}

@article{HoweTanWillenbring2008, 
title={The Stability of Graded Multiplicity in the Setting of the Kostant-Rallis Theorem}, 
volume={13}, 
number={3}, 
journal={Transformation Groups}, 
author={Howe, Roger and Tan, Eng-Chye and Willenbring, J.}, 
year={2008}}

@article{HoweTanWillenbring2005, 
title={Stable branching rules for classical symmetric pairs}, 
journal={Transactions of the American Mathematical Society}, 
author={Howe, Roger and Tan, Eng-Chye and Willenbring, J.}, 
year={2005}}

@article {JohnsonWallach1977,
    AUTHOR = {Johnson, Kenneth D. and Wallach, Nolan R.},
     TITLE = {Composition series and intertwining operators for the
              spherical principal series. {I}},
   JOURNAL = {Trans. Amer. Math. Soc.},
  FJOURNAL = {Transactions of the American Mathematical Society},
    VOLUME = {229},
      YEAR = {1977},
     PAGES = {137--173}
}

@article{Kostant1963,
 author = {Bertram Kostant},
 journal = {American Journal of Mathematics},
 number = {3},
 pages = {327--404},
 publisher = {Johns Hopkins University Press},
 title = {Lie Group Representations on Polynomial Rings},
 volume = {85},
 year = {1963}
}

@article{KostantRallis1971,
 author = {B. Kostant and S. Rallis},
 journal = {American Journal of Mathematics},
 number = {3},
 pages = {753--809},
 publisher = {Johns Hopkins University Press},
 title = {Orbits and Representations Associated with Symmetric Spaces},
 volume = {93},
 year = {1971}
}

@article{KwonJang2021,
address = {Orlando, Fla. :},
author = {Jang, Il-Seung and Kwon, Jae-Hoon},
journal = {Journal of combinatorial theory.},
title = {Flagged Littlewood-Richardson tableaux and branching rule for classical groups},
publisher = {Academic Press},
volume = {181},
year = {2021},
}

@article{LecouveyLenart2020,
    author = {Lecouvey, Cédric and Lenart, Cristian},
    title = "{Combinatorics of Generalized Exponents}",
    journal = {International Mathematics Research Notices},
    volume = {2020},
    number = {16},
    pages = {4942-4992},
    year = {2018},
    month = {07},
}

@article{Littlewood1944,
 author = {D. E. Littlewood},
 journal = {Philosophical Transactions of the Royal Society of London. Series A, Mathematical and Physical Sciences},
 number = {809},
 pages = {387--417},
 publisher = {The Royal Society},
 title = {On Invariant Theory under Restricted Groups},
 volume = {239},
 year = {1944}
}

@book{Littlewood1945, place={Oxford}, title={Theory of Group Characters}, publisher={Clarendon Press}, author={D. E. Littlewood}, year={1945}}

@incollection {NelsenRam2003,
    AUTHOR = {Nelsen, Kendra and Ram, Arun},
     TITLE = {Kostka-{F}oulkes polynomials and {M}acdonald spherical
              functions},
 BOOKTITLE = {Surveys in combinatorics, 2003 ({B}angor)},
    SERIES = {London Math. Soc. Lecture Note Ser.},
    VOLUME = {307},
     PAGES = {325--370},
 PUBLISHER = {Cambridge Univ. Press, Cambridge},
      YEAR = {2003}
}

@article{STEMBRIDGE198779,
title = {Rational tableaux and the tensor algebra of gln},
journal = {Journal of Combinatorial Theory, Series A},
volume = {46},
number = {1},
pages = {79-120},
year = {1987},
author = {John R Stembridge},
}

@Inbook{Wallach2017,
author="Wallach, Nolan R.",
editor="Cogdell, Jim
and Kim, Ju-Lee
and Zhu, Chen-Bo",
title="An Analogue of the Kostant--Rallis Multiplicity Theorem for $\theta$-Group Harmonics",
bookTitle="Representation Theory, Number Theory, and Invariant Theory: In Honor of Roger Howe on the Occasion of His 70th Birthday",
year="2017",
publisher="Springer International Publishing",
pages="603--626",
}

@book {Wallach2017-2,
    AUTHOR = {Wallach, Nolan R.},
     TITLE = {Geometric invariant theory},
    SERIES = {Universitext},
      NOTE = {Over the real and complex numbers},
 PUBLISHER = {Springer, Cham},
      YEAR = {2017},
}

@article{WallachWillenbring2000, 
title={On Some $q$-Analogs of a Theorem of Kostant-Rallis}, volume={52}, 
number={2}, 
journal={Canadian Journal of Mathematics}, 
publisher={Cambridge University Press}, 
author={Wallach, N. R. and Willenbring, J.}, 
year={2000}, 
pages={438–448}}

@article{Willenbring2002, 
title={An application of the Littlewood restriction formula to the Kostant-Rallis Theorem}, 
journal={Transactions of the American Mathematical Society}, 
author={Willenbring, J.}, 
year={2002}}

@incollection {WillenbringVanGroningen2014,
    AUTHOR = {van Groningen, Anthony and Willenbring, Jeb F.},
     TITLE = {The cubic, the quartic, and the exceptional group {$\rm G_2$}},
 BOOKTITLE = {Developments and retrospectives in {L}ie theory},
    SERIES = {Dev. Math.},
    VOLUME = {38},
     PAGES = {385--397},
 PUBLISHER = {Springer, Cham},
      YEAR = {2014},
   MRCLASS = {20G05 (17B10 17B20 20G41)},
  MRNUMBER = {3308792},
}
    
 \end{document}